\documentclass[a4paper,10pt]{article}

\usepackage[square,numbers]{natbib}
\usepackage{amsfonts,amsmath,amssymb,amsthm}
\usepackage{hyperref}
\usepackage{graphicx}
\usepackage{tikz}
\usepackage{mathpazo}
\usetikzlibrary{arrows}

\theoremstyle{definition}
\newtheorem{theorem}{Theorem}[section]
\newtheorem{lemma}[theorem]{Lemma}
\newtheorem{prop}[theorem]{Proposition}

\newtheorem{que}[theorem]{Question}

\title{Diophantine equations via cluster transformations}
\author{Philipp Lampe}

\begin{document}
\maketitle
\abstract{Motivated by Fomin-Zelevinsky's theory of cluster algebras we introduce a variant of the Markov equation; we show that all natural solutions of the equation arise from an initial solution by cluster transformations.}
\tableofcontents

\section{Cluster algebras}

\subsection{Introduction}

Fomin and Zelevinsky's cluster theory provides a common combinatorial framework for problems in representation theory, Lie theory, hyperbolic geometry and mathematical physics. The theory was initiated in a series of four influential papers \cite{FZ,FZ2,BFZ,FZ4}. The key notion to define a cluster algebra is the so-called \textit{mutation}, which we will recall in the next section. We can mutate \textit{quivers} and \textit{cluster variables}. Surprisingly, cluster algebra mutations describe interesting phenomena in various branches of mathematics. For example, a correspondence between non-initial cluster variables and positive roots in a certain root system establishes a link to Lie theory. There are deeper links to Lie theory. In fact, a conjectural correspondence between cluster monomials and Lusztig's canonical basis elements was one of the original motivations to introduce cluster algebras. For a different example, the Caldero-Chapoton map \cite{CC} links cluster algebras with quiver representations. In this context, mutation is related to tilting. Cluster algebras also occur in hyperbolic geometry, where mutation is related to Ptolemy's theorem. Some cluster algebras admit invariants of mutation, which play an important role in the context of dynamical systems. 

Markov's equation provides a common number-theoretic framework for problems in representation theory, geometry and arithmetic. The key notion to solve the Markov equation is sometimes called \textit{Vieta jumping}. We will recall it in Section \ref{SectionMarkov}. Although we can solve the Markov equation by elementary methods, it describes interesting phenomena in various branches in mathematics. For example, Gorodentsev-Rudakov \cite{GR} show that the solutions of the Markov equation describe ranks of exceptional vector bundles on projective spaces.

In this article, we wish to show how cluster mutations can generate all solutions of a particular Diophantine equation. The form of the Diophantine equation is related to a Laurent polynomial in an upper cluster algebra, which is invariant under mutation. In Section \ref{SectionRudiments} we recall the basic notions of cluster theory. Especially, we give a formal definition of the aforementioned mutation process and present its main features. In the following Section \ref{SectionSurface} we introduce Fomin-Shapiro-Thurston's surface cluster algebras \cite{FST}. The construction is based on work of Fock-Goncharov \cite{FG1,FG2} and Gekhtman-Shapiro-Vainshtein \cite{GSV}. Surface cluster algebras provide interesting instances of cluster algebras with invariants of mutation. Section \ref{SectionMarkov} introduces and solves Markov's equation and interprets it from a cluster theoretic point of view. We introduce a new Diophantine equation in Section \ref{SectionVariant} and show that we can solve the equation by cluster transformations. In Section \ref{SectionTorusMinus} we study another cluster algebra with an invariant and ask some questions about the corresponding Diophantine equation.

\subsection{Rudiments of cluster algebras}
\label{SectionRudiments}

Let us briefly recall the definition of a cluster algebra. For the rest of the section, we fix a positive integer $n\in\mathbb{N}$. For brevity, we write $[n]$ instead of $\{1,2,\ldots,n\}$. 

A key notion of cluster theory is the \textit{mutation} of \textit{skew-symmetrizable matrices}. Here an $n\times n$ integer matrix $B$ is called \textit{left skew-symmetrizable} if there exists an $n\times n$ diagonal matrix $D=\operatorname{diag}(d_1,\ldots,d_n)$ with $d_1,\ldots,d_n>0$ such that the matrix $DB$ is skew-symmetric. We define \textit{right skew-symmetrizability} analogously. An easy proposition asserts that $B$ is left skew-symmetrizable if and only if it is right skew-symmetrizable. In this case we say that $B$ is \textit{skew-symmetrizable}. Two skew-symmetrizable $n\times n$ matrices $B=(b_{ij})_{i,j\in [n]}$ and $C=(c_{ij})_{i,j\in[n]}$ are called \textit{isomorphic} if there exists a bijection $\sigma\in S_n$ such that $c_{ij}=b_{\sigma(i),\sigma(j)}$ for all $i,j$. In this case we write $B\simeq C$.   

Every skew-symmetric matrix is skew-symmetrizable with $D=I_n$. We can interpret a skew-symmetric matrix as the \textit{signed adjacency matrix} of a \textit{quiver} without loops and $2$-cycles. Here, a quiver is a finite directed graph, i.\,e. a quadruple $Q=(Q_0,Q_1,s,t)$ consisting of a finite set $Q_0$ of vertices, a finite set $Q_1$ of arrows, and maps $s,t\colon Q_1\to Q_0$ that assign to every arrow its starting and its terminating vertex. Its signed adjacency matrix is the integer square matrix $B=B(Q)=(b_{ij})_{i,j\in Q_0}$ with entries $b_{ij}=a_{ij}-a_{ji}$ where $a_{ij}$ is the number of arrows that start in $s(\alpha)=i$ and terminate in $t(\alpha)=j$. By definition, the signed adjacency matrix is skew-symmetric. Let us assume that $Q$ does not contain closed paths of length $2$ (and especially no loops). In this case we can recover $Q$ from its signed adjacency matrix up to isomorphism. More precisely, two quivers $Q$ and $Q'$ satisfy $B(Q)\simeq B(Q')$ if and only $Q\simeq Q'$.

The \textit{mutation} of a skew-symmetrizable integer matrix $B=(b_{ij})_{i,j\in[n]}$ at an index $k\in [n]$ is a new matrix $\mu_k(B)=B'=(b_{ij}')_{i,j\in[n]}$ whose entries are given by the formula 
\begin{align*}
b'_{ij}=\begin{cases}-b_{ij},&\textrm{if \ } k\in\{i,j\};\\ b_{ij}+\frac{1}{2}(b_{ik}\vert b_{kj}\vert+\vert b_{ik}\vert b_{kj}), &\textrm{if \ } k\notin\{i,j\}.\end{cases}
\end{align*} 
Note that $B'$ is again skew-symmetrizable with the same skew-symmetrizer. Moreover, an easy argument shows that mutation is involutive, i.\,e. we have $(\mu_k\circ\mu_k)(B)=B$ for all $k\in[n]$. \textit{Mutation equivalence} is the smallest equivalence relation on the set of skew-symmetrizable integer $n\times n$ matrices that is closed under isomorphism and mutation. In other words, we call two skew-symmetrizable integer $n\times n$ matrices $B$ and $B'$ mutation equivalent if there exists a sequence $(k_1,\ldots,k_r)\in[n]^r$ of length $r\geq 0$ such that $(\mu_{k_r}\circ\cdots\circ\mu_{k_1})(B)\simeq B'$. The matrix $B$ is called \textit{mutation finite} if its mutation equivalence class is finite.

Let $K$ be a field of characteristic $0$. A \textit{cluster} is a tuple $\mathbf{x}=(x_i)_{i\in[n]}$ of algebraically independent variables over $K$. An element $x_i$ of a cluster $(x_i)_{i\in [n]}$ is called \textit{cluster variable}. A \textit{seed} is a pair $(B,\mathbf{x})$ where $B$ is a skew-symmetrizable integer $n\times n$ matrix and $\mathbf{x}=(x_i)_{i\in [n]}$ is a cluster. The \textit{mutation} of a seed $(B,\mathbf{x})$ at a vertex $k\in [n]$ is a new seed $\mu_k(B,\mathbf{x})=(B',\mathbf{x}')$ where $B'=\mu_k(B)$ and $\mathbf{x}'$ is obtained from $\mathbf{x}$ by replacing $x_k$ with
\begin{align*}
x_k'=\frac{1}{x_k}\left(\prod_{i\in [n]\colon b_{ik}>0}x_i^{b_{ik}}+\prod_{i\in [n]\colon b_{ik}<0}x_i^{-b_{ik}}\right)\in K(x_i\colon i\in [n]).
\end{align*}
It is easy to see that $\mathbf{x}'$ is algebraically independent so that $(B',\mathbf{x}')$ is again a seed. The above equation is sometimes called \textit{exchange relation}. As before, mutation is involutive, i.\,e. we have $(\mu_k\circ\mu_k)(B,\mathbf{x})=(B,\mathbf{x})$ for all $k\in [n]$. Two seeds $(B,\mathbf{x})$ and $(C,\mathbf{y})$ are called \textit{isomorphic} if there exists a bijection $\sigma\in S_n$ such that $c_{ij}=b_{\sigma(i),\sigma(j)}$ for all $i,j\in [n]$ and $x_k=y_{\sigma(k)}$ for all $k\in [n]$. In this case we write $(B,\mathbf{x})\simeq(C,\mathbf{y})$. \textit{Mutation equivalence} is the smallest equivalence relation on the class of seeds that is closed under isomorphism and mutation. In other words, we call two seeds $(B,\mathbf{x})$ and $(C,\mathbf{y})$ mutation equivalent if there exists a sequence $(k_1,\ldots,k_r)\in [n]^r$ of length $r\geq 0$ such that $(\mu_{k_r}\circ\cdots\circ\mu_{k_1})(B,\mathbf{x})\simeq (C,\mathbf{y})$.

Let $(B,\mathbf{x})$ be a seed. The field $\mathcal{F}=K(x_i^{\pm} \colon i\in [n])$ is called the \textit{ambient field}. The definition of mutation implies that the ambient field contains the union 
\begin{align*}
\chi(B,\mathbf{x})=\bigcup_{(C,\mathbf{y})\simeq(B,\mathbf{x})} \{y_i\colon i\in [n]\}
\end{align*}
of all clusters in all seeds in the mutation equivalence class of $(B,\mathbf{x})$. We define the \textit{cluster algebra} $\mathcal{A}(B,\mathbf{x})$ to be the $K$-subalgebra of $\mathcal{F}$ generated by $\chi(B,\mathbf{x})$. It follows from the definition that $\mathcal{A}(B,\mathbf{x})=\mathcal{A}(C,\mathbf{y})$ if $(B,\mathbf{x})\simeq(C,\mathbf{y})$. Therefore, we call a seed $(C,\mathbf{y})$ in the mutation equivalence class of $(B,\mathbf{x})$ a \textit{seed of} $\mathcal{A}(B,\mathbf{x})$, the element $\mathbf{y}$ a \textit{cluster of} $\mathcal{A}(B,\mathbf{x})$, and an element in $\chi(B,\mathbf{x})$ a \textit{cluster variable of} $\mathcal{A}(B,\mathbf{x})$. A monomial in the cluster variables of a single cluster of a cluster algebra is called \textit{cluster monomial}. Sometimes we write $\mathcal{A}(B)$ instead of $\mathcal{A}(B,\mathbf{x})$, because there is an algebra isomorphism $\mathcal{A}(B,\mathbf{x})\cong\mathcal{A}(B,\mathbf{y})$ for all clusters $\mathbf{x},\mathbf{y}$. We also write $\mathcal{A}(Q,\mathbf{x})$ instead of $\mathcal{A}(B(Q),\mathbf{x})$ when we construct the matrix $B$ from the quiver $Q$. Note that $\mathcal{A}(B)=\mathcal{A}(-B)$ for all skew-symmetrizable matrices $B$. Furthermore, we refer to the integer $n$ as the \textit{rank} of the cluster algebra $\mathcal{A}(B,\mathbf{x})$.

Fomin and Zelevinsky prove two main theorems about cluster algebras. The first theorem is called Laurent phenomenon \cite{FZ}. It asserts that every cluster variable of $\mathcal{A}(B,\mathbf{x})$ is a Laurent polynomial in $\mathbf{x}$ with integer coefficients. Especially, the cluster algebra $\mathcal{A}(B,\mathbf{x})\subseteq K[x_i^{\pm 1}\colon i\in [n]]$ is a subalgeba of the algebra of Laurent polynomials in $\mathbf{x}$. The second main theorem is the classification of cluster algebras with finitely many cluster variables by finite type roots systems \cite{FZ2}. More precisely, a cluster algebra $\mathcal{A}(B)$ admits only finitely many cluster variables if and only if $B$ is mutation equivalent to a matrix whose Cartan counterpart is the Cartan matrix of a root system of finite type. For example, the cluster algebra $\mathcal{A}(Q)$ attached to a connected quiver $Q$ admits only finitely many cluster variables if and only if $Q$ is mutation equivalent to an orientation of a Dynkin diagram of type $A$, $D$, or $E$. (The Dynkin diagrams of type $B$, $C$, $F$, and $G$ arise from cluster algebras with skew-symmetrizable $B$-matrices.) In this case, thanks to the Laurent phenomenon a non-initial cluster variable may be written as $P/(\prod_{i\in Q_0}x_i^{a_i})$ for some polynomial $P\in\mathbb{Z}[x_i\colon i\in [n]]$ and some natural numbers $a_i\in\mathbb{N}$. Then the sums $\sum_{i\in Q_0}a_i\alpha_i$ are precisely the positive roots in the corresponding root system written as a linear combination in the simple roots $\alpha_i$. It follows from this description that if $\mathcal{A}(Q,\mathbf{x})$ has finitely many cluster variables, then $Q$ is mutation finite.

Sometimes we \textit{freeze} quiver vertices or matrix indices. In the context of quivers this means, we partition the set of vertices $$Q_0=Q_0^{(mu)}\sqcup Q_0^{(fr)}$$ in two sets called \textit{mutable} and \textit{frozen} vertices and allow only mutations at mutable vertices. In this way we obtain a smaller generating set of cluster variables for the cluster algebra $\mathcal{A}(Q)$. The variables attached to the frozen vertices are called \textit{frozen variables}. Fomin-Zelevinsky's classification \cite{FZ4} generalizes to quivers with frozen vertices. More precisely, a cluster algebra $\mathcal{A}(Q)$ attached to a connected quiver $Q$ with frozen vertices admits only finitely many cluster variables if and only if the full subquiver of $Q$ on the mutable vertices is mutation equivalent to an orientation of a Dynkin diagram of type $A$, $D$, or $E$.   

Grabowski \cite{G} constructs gradings of cluster algebras. More precisely, if $v\in\mathbb{Z}^n$ is a solution to the equation $v^TB=0$, then we can equip $\mathcal{A}(B,\mathbf{x})$ with a grading such that $\operatorname{deg}(x_i)=v_i$ for all $i$.

Let $(B,\mathbf{x})$ be a seed. Berenstein-Fomin-Zelevinsky introduce the \textit{upper cluster algebra} $\overline{\mathcal{A}}(B,\mathbf{x})$ as the intersection of Laurent polynomial rings
\begin{align*}
\overline{\mathcal{A}}(B,\mathbf{x})=\bigcap_{(C,\mathbf{y})\simeq(B,\mathbf{x})} K[y_i^{\pm 1}\colon i\in [n]]\subseteq K(x_i^{\pm 1}\colon i\in [n])=\mathcal{F}.
\end{align*}
The Laurent phenomenon implies $\mathcal{A}(B,\mathbf{x})\subseteq\overline{\mathcal{A}}(B,\mathbf{x})$.

\subsection{Surface cluster algebras}
\label{SectionSurface}

In this section we briefly describe how Fomin-Shapiro-Thurston \cite{FST} (generalizing work of Fock-Goncharov \cite{FG1,FG2} and Gekhtman-Shapiro-Vainshtein \cite{GSV}) associate \textit{a cluster algebra} $\mathcal{A}=\mathcal{A}(\Sigma,M)$ with a bordered surface $\Sigma$ with marked points $M$. This class of cluster algebras is interesting for several reasons. For example, the construction always yields mutation finite cluster algebras, see Fomin-Shapiro-Thurston \cite{FST}. In fact, by a theorem of Felikson-Shapiro-Tumarkin \cite{FST2} almost all mutation finite quivers of large rank arise in this way. We start with basic notions.

A \textit{bordered surface} is a compact, connected, oriented, $2$-dimensional Riemann surface $\Sigma$ with or without boundary. Moreover, a \textit{bordered surface with marked points} is a pair $(\Sigma,M)$ consisting of a bordered surface $\Sigma$ together with a finite set $M\subseteq\Sigma$ such that every connected component of the boundary $\partial \Sigma$ contains at least one point of $M$. In this case, elements in the set $M$ are called \textit{marked points} and marked points in the interior of $M$ are called \textit{punctures}. A \textit{homeomorphism} between two bordered surfaces with marked points $(\Sigma,M)$ and $(\Sigma',M')$ is a homeomorphism $\phi\colon\overline{\Sigma}\to\overline{\Sigma'}$ such that $\phi\vert_{M}\colon M\to M'$ is a bijection. In this case,  $(\Sigma,M)$ and $(\Sigma',M')$ are called \textit{homeomorphic}. A typical example of a bordered surface with marked points is a regular $n$-gon together with its set of vertices; it is homeomorphic to a disk with $n$ marked points on the boundary. We refer to the cases $n=1$ and $n=2$ as a monogon or a bigon, respectively.

Let us assume that $(\Sigma,M)$ is a bordered surface with marked points, which, for a technical reason, is not homeomorphic to a sphere with exactly one puncture, a sphere with exactly two punctures, a sphere with exactly three punctures, a monogon without a puncture, a monogon with exactly one puncture, a bigon without a puncture, or a triangle without a puncture. Note that, up to homeomorphism the bordered surface $(\Sigma, M)$ is determined by the genus $g$ of $\Sigma$, the number $p$ of punctures, the number $b$ of boundary components and the sequence $(n_1,\ldots,n_b)$ of numbers of marked points on the boundary components.

An \textit{arc} is a curve $\gamma\colon[0,1]\to \Sigma$ such that the endpoints $\gamma(0),\gamma(1)$ lie in $ M$, the restriction $\gamma\vert_{(0,1)}\colon (0,1)\to \Sigma$ is injective and its image is disjoint from $M\cup\partial \Sigma$, and it does not cut out a monogon without punctures or a bigon without punctures. Two arcs $\gamma,\gamma'$ are called \textit{compatible} if the images $\gamma(0,1)$ and $\gamma'(0,1)$ are disjoint, i.\,e. the arcs are disjoint except for possible intersections at the endpoints. Two arcs $\gamma,\gamma'$ are called \textit{isotopic} if there is a homotopy $H\colon [0,1]\times [0,1]\to\Sigma$ from $\gamma$ to $\gamma'$, i.\,e. a map with $H(0,-)=\gamma$ and $H(1,-)=\gamma'$, such that $H(t,-)\colon [0,1]\to \Sigma$ is an arc for all $t\in [0,1]$.

A \textit{triangulation} $\mathcal{T}$ of $(\Sigma,M)$ is a maximal collection of pairwise compatible and not isotopic arcs. Two triangulations $\mathcal{T}$ and $\mathcal{T}'$ are called \textit{equivalent} if there exists a map $\pi\colon[0,1]\times\Sigma\to\Sigma$ with $\pi(0,\mathcal{T})=\mathcal{T}$ and $\pi(1,\mathcal{T})=\mathcal{T}'$ such that $\pi(t,-)$ induces a homeomorphism of bordered surfaces with marked points for all $t\in[0,1]$. An arc is called a \textit{boundary arc} if it is isotopic to an arc in the boundary $\partial \Sigma$. Otherwise it is called \textit{flippable arc}. Flippable arcs have a remarkable property. If a flippable arc is part of a triangulation $\mathcal{T}$, then there exists an arc $\gamma'$ such that $\gamma'$ is not isotopic to $\gamma$ and $\mathcal{T}'=(\mathcal{T}\backslash\{\gamma\})\cup\{\gamma'\}$ is a triangulation of $(\Sigma,M)$. The conditions determine the arc $\gamma'$ uniquely up to isotopy. The triangulation $\mathcal{T}'$ is called the \textit{flip} of $\mathcal{T}$ at $\gamma$ and is denoted by $\mu_{\gamma}(\mathcal{T})$.

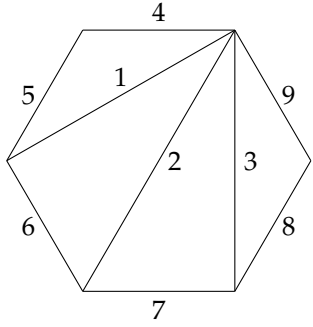
\begin{figure}
\begin{center}
\begin{tikzpicture}[auto,scale=2.0]
    \path[-] (1,0) edge node[above,right] {$9$} (0.5,0.866);
    \path[-] (0.5,0.866) edge node[above] {$4$} (-0.5,0.866);
    \path[-] (-0.5,0.866) edge node[above,left] {$5$} (-1,0);
    \path[-] (-1,0) edge node[below,left] {$6$} (-0.5,-0.866);
    \path[-] (-0.5,-0.866) edge node[below] {$7$} (0.5,-.866);
    \path[-] (0.5,-.866) edge node[below,right] {$8$} (1,0);
    
    \path[-] (-1,0) edge node[above] {$1$} (0.5,0.866);
    \path[-] (-0.5,-0.866) edge node[right] {$2$} (0.5,0.866);
    \path[-] (0.5,-0.866) edge node[right] {$3$} (0.5,0.866);
\end{tikzpicture}
\end{center}
\label{Figure:A}
\caption{A triangulation of a hexagon}
\end{figure}

Now we are ready to construct the cluster algebra $\mathcal{A}(\Sigma,M)$. On a combinatorial level, the seeds of the cluster algebra $\mathcal{A}(\Sigma,M)$ correspond to (equivalence classes of) triangulations of $(\Sigma,M)$. Let $\mathcal{T}$ be a triangulation of $(\Sigma,M)$ and let us denote its seed by $(Q(\mathcal{T}),\mathbf{x}(\mathcal{T}))$. The mutable vertices in the quiver $Q(\mathcal{T})$ are defined to be the (isotopy classes of) flippable arcs in $\mathcal{T}$, whereas its frozen vertices are defined to be the (isotopy classes of) boundary arcs in $\mathcal{T}$. The arrows in $Q(\mathcal{T})$ are constructed using orientation of the triangles in the triangulation. More precisely, for a triangle which is bounded by arcs $\gamma_1,\gamma_2,\gamma_3$ in this order we introduce arrows $\gamma_1\to\gamma_2$, $\gamma_2\to\gamma_3$ and $\gamma_3\to\gamma_1$. Flips of triangulations correspond to mutations of quivers. More precisely, a proposition asserts that for every flippable arc $\gamma$ in a triangulation we have $Q(\mu_{\gamma}(\mathcal{T}))=\mu_{\gamma}(Q(\mathcal{T}))$. The cluster exchange relations have an interpretation as \textit{Ptolemy relations}. Here, we think of a cluster or frozen variable $x_{\gamma}$ as the \textit{lambda length} of the arc $\gamma$. The lambda lengths satisfy a remarkable theorem which we may view as an analogue of Ptolemy's theorem in Euclidean geometry: assume that a quadrangle is bounded by arcs $\alpha$, $\beta$, $\gamma$ and $\delta$ and that its two diagonal arcs are $\epsilon$ and $\digamma$. Then we have $x_{\alpha}x_{\gamma}+x_{\beta}x_{\delta}=x_{\epsilon}x_{\digamma}$. The equation is precisely the exchange relation in the cluster algebra.

An example is the disk $\Sigma=\mathcal{D}^1$ with $n+3$ points on the boundary. A triangulation consists of $n+3$ boundary arcs and $n$ flippable arcs. The cluster algebra has finitely many cluster variables, because $(\Sigma,M)$ admits only finitely many arcs. In fact, it is of type $A_n$, because it admits several triangulations whose quivers are orientations of Dynkin diagrams of type $A_n$. Figure \ref{Figure:A} shows a disk with $6$ marked points on the boundary. The number of triangulations is equal to the Catalan number $C_{n}=\frac{1}{n+2}\binom{2n+2}{n+1}=14$. The quiver of the chosen triangulation is equal to $1\to2\to3$.

\section{Solving Diophantine equations by cluster transformations}

\subsection{The Markov equation}
\label{SectionMarkov}

In this section we want to study the cluster algebra attached to the torus with exactly one marked point in the interior more closely. It is related to the \textit{Markov equation}. The material in this section is classical and the number theoretic part goes back to Markov \cite{M}. For the relation to cluster algebras see for example Peng-Zhang \cite{PZ}.

Figure \ref{Figure:TriangulationMarkov} shows the universal cover of the surface together with a triangulation $\mathcal{T}$ made of two triangles. If we label the arcs of this triangulation by $1$, $2$, $3$, then the associated quiver $Q=Q(\mathcal{T})$ has two arrows $1\rightrightarrows 2$, two arrows $2\rightrightarrows 3$ and two arrows $3\rightrightarrows 1$. The corresponding signed adjacency matrix is
\begin{align*}
B=B(Q)=\left(\begin{matrix}
0&2&-2\\
-2&0&2\\
2&-2&0
\end{matrix}\right).
\end{align*}
The following proposition is immediate.

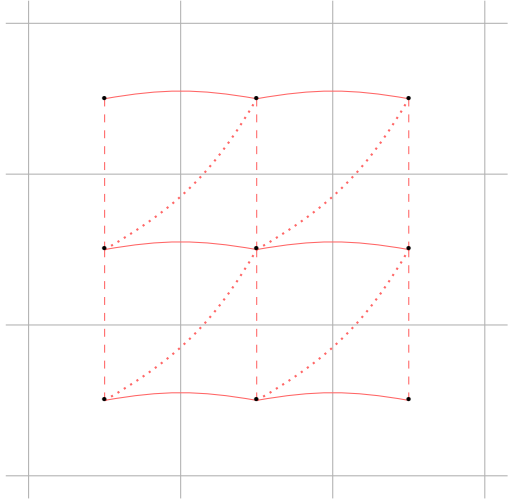
\begin{figure}[h!]
\begin{center}
\begin{tikzpicture}

\newcommand\grid{1}
\newcommand\uber{0.3}

\draw[color=black!30!] (0,0-\uber) to (0,6*\grid+\uber);
\draw[color=black!30!] (2*\grid,0-\uber) to (2*\grid,6*\grid+\uber);
\draw[color=black!30!] (4*\grid,0-\uber) to (4*\grid,6*\grid+\uber);
\draw[color=black!30!] (6*\grid,0-\uber) to (6*\grid,6*\grid+\uber);
\draw[color=black!30!] (0-\uber,0) to (6*\grid+\uber,0);
\draw[color=black!30!] (0-\uber,2*\grid) to (6*\grid+\uber,2*\grid);
\draw[color=black!30!] (0-\uber,4*\grid) to (6*\grid+\uber,4*\grid);
\draw[color=black!30!] (0-\uber,6*\grid) to (6*\grid+\uber,6*\grid);

\coordinate(AA) at (\grid,\grid) {};
\coordinate(AB) at (3*\grid,\grid) {};
\coordinate(AC) at (5*\grid,\grid) {};
\coordinate(BA) at (\grid,3*\grid) {};
\coordinate(BB) at (3*\grid,3*\grid) {};
\coordinate(BC) at (5*\grid,3*\grid) {};
\coordinate(CA) at (\grid,5*\grid) {};
\coordinate(CB) at (3*\grid,5*\grid) {};
\coordinate(CC) at (5*\grid,5*\grid) {};

\draw[bend left=10,color=red!60!] (AA) to (AB);
\draw[bend left=10,color=red!60!] (AB) to (AC);
\draw[bend left=10,color=red!60!] (BA) to (BB);
\draw[bend left=10,color=red!60!] (BB) to (BC);
\draw[bend left=10,color=red!60!] (CA) to (CB);
\draw[bend left=10,color=red!60!] (CB) to (CC);

\draw[bend right=15,color=red!60!,dotted,thick] (AA) to (BB);
\draw[bend right=15,color=red!60!,dotted,thick] (BB) to (CC);
\draw[bend right=15,color=red!60!,dotted,thick] (BA) to (CB);
\draw[bend right=15,color=red!60!,dotted,thick] (AB) to (BC);

\draw[color=red!60!,dashed] (AA) to (BA);
\draw[color=red!60!,dashed] (BA) to (CA);
\draw[color=red!60!,dashed] (AB) to (BB);
\draw[color=red!60!,dashed] (BB) to (CB);
\draw[color=red!60!,dashed] (AC) to (BC);
\draw[color=red!60!,dashed] (BC) to (CC);

\node at (AA) {\Large{$\cdot$}};
\node at (AB) {\Large{$\cdot$}};
\node at (AC) {\Large{$\cdot$}};
\node at (BA) {\Large{$\cdot$}};
\node at (BB) {\Large{$\cdot$}};
\node at (BC) {\Large{$\cdot$}};
\node at (CA) {\Large{$\cdot$}};
\node at (CB) {\Large{$\cdot$}};
\node at (CC) {\Large{$\cdot$}};

\end{tikzpicture}
\end{center}
\caption{A triangulation of the torus with one marked point}
\label{Figure:TriangulationMarkov}
\end{figure}

\begin{prop} 
\label{MutationClassMarkov}
For every $k\in\{1,2,3\}$ we have $\mu_1(Q)\simeq\mu_2(Q)\simeq\mu_3(Q)\simeq Q$. Especially, the mutation equivalence class of $Q$ is a singleton.
\end{prop}

Let $\mathbf{x}=(x_1,x_2,x_3)$ be the initial cluster, so that the cluster variables $x_1,x_2,x_3$ correspond to the arcs $1,2,3$. Assume that the seed $(R,\mathbf{y})$ is mutation equivalent to $(Q,\mathbf{x})$. Proposition \ref{MutationClassMarkov} implies that the mutation $\mu_i$ (with $1\leq i\leq 3$) exchanges the cluster variable $y_i$ with the cluster variable $y_i'$ given by equations
\begin{align*}
&y_1y_1'=y_2^2+y_3^2,\\
&y_2y_2'=y_3^2+y_1^2,\\
&y_3y_3'=y_1^2+y_2^2.
\end{align*}
We can extend the $\mu_i$ to functions $(\mathcal{F}^{\times})^3\to(\mathcal{F}^{\times})^3$ by setting 
\begin{align*}
\mu_1(a,b,c)=(\tfrac{b^2+c^2}{a},b,c),&&\mu_2(a,b,c)=(a,\tfrac{c^2+a^2}{b},c),&&\mu_3(a,b,c)=(a,b,\tfrac{a^2+b^2}{c})
\end{align*}
for all $(a,b,c)\in(\mathcal{F}^{\times})^3$. Note that the functions $\mu_i$ are involutions. 

Grabowski's grading equation $v^TB=0$ admits a unique solution $v^T=(1,1,1)$ up to a scalar multiple. By Proposition \ref{MutationClassMarkov}, the vector $v$ is a unique solution to the grading equation for every signed adjacency matrix in a seed of $\mathcal{A}(Q,\mathbf{x})$. That means, $\mathcal{A}(Q,\mathbf{x})$ is a graded algebra when we define the degree of every cluster variable to be equal to $1$. For example, the above exchange relations are homogeneous of degree $2$. Especially, $\mathcal{A}(Q,\mathbf{x})$ is a \textit{positively} graded algebra.

The next proposition concerns an invariant of the mutation rule. Define a map $T\colon(\mathcal{F}^{\times})^3\to\mathcal{F}$ by 
\begin{align*}
T(a,b,c)=\frac{a^2+b^2+c^2}{abc}
\end{align*}for all $(a,b,c)\in(\mathcal{F}^{\times})^3$.  

\begin{prop}
\label{Invariant}
We have $(T\circ\mu_i)(a,b,c)=T(a,b,c)$ for all $i\in\{1,2,3\}$ and $a,b,c\in\mathcal{F}^{\times}$. 
\end{prop}

\begin{proof}
Without loss of generality we may assume $i=1$. Let $a,b,c\in\mathcal{F}$ be non-zero elements. Put $T=T(a,b,c)$. By assumption, $a$ is a zero of the polynomial $\lambda^2-bcT\lambda+b^2+c^2\in\mathcal{F}[\lambda]$. By Vieta's formula, $a'=\tfrac{b^2+c^2}{a}$ is another zero of the polynomial, so that $T(a',b,c)=T$.
\end{proof}

The proof shows that the mutation rule $aa'=b^2+c^2$ may be seen as Vieta's formula. It may be replaced by Vieta's formula $a+a'=bcT$. Especially, if $a,b,c,T\in\mathbb{N}\subseteq\mathcal{F}$ are positive integers, then so is $a'$. 

Let $\mathbf{x}=(x_1,x_2,x_3)$ be the initial cluster associated with the triangulation $\mathcal{T}$. Propositions \ref{MutationClassMarkov} and \ref{Invariant} imply that $T$ is invariant under mutation. That is, if $(R,\mathbf{y})$ is another seed of $\mathcal{A}(Q,\mathbf{x})$, then $T(\mathbf{x})=T(\mathbf{y})$. Berenstein-Fomin-Zelevinsky \cite{BFZ} use the element $T_0=T(\mathbf{x})$ and the grading to show that $\mathcal{A}(Q,\mathbf{x})\neq\overline{\mathcal{A}}(Q,\mathbf{x})$. In fact, $T_0\in \overline{\mathcal{A}}(Q,\mathbf{x})\backslash \mathcal{A}(Q,\mathbf{x})$, because $T_0=T(\mathbf{y})$ is a Laurent polynomial every cluster $\mathbf{y}$, but is not contained in the positively graded algebra $\mathcal{A}(Q,\mathbf{x})$ for degree reasons. Moreover, Keller \cite{Ke} uses the grading together with Cerulli--Irelli-Keller-Labardini--Fragoso-Plamondon's \cite{CKLP} linear independence of cluster monomials to show that $\mathcal{A}(Q,\mathbf{x})$ is not noetherian.

The \textit{Markov equation} asks for all triples $(a,b,c)$ of positive integers with $T(a,b,c)=3$. Equivalently, the Markov equation is the Diophantine equation $a^2+b^2+c^2=3abc$. The solution $(a,b,c)=(1,1,1)$ is called the \textit{fundamental solution}. Markov's theorem asserts all solutions can be obtained from the fundamental solution by a sequence of mutations.

\begin{theorem}[Markov]
Assume that $a,b,c$ are positive integers with $T(a,b,c)=3$. There exists a sequence $(k_1,\ldots,k_r)\in\{1,2,3\}^r$ of length $r\geq 0$ such that $(a,b,c)=(\mu_{k_r}\circ\ldots\circ\mu_{k_1})(1,1,1)$. 
\end{theorem}

\begin{proof}
We prove the theorem by induction on the maximum $m$ of $a,b,c$. The claim is true for $m=1$. Assume that $m>1$. Without loss of generality we may assume $a\geq b\geq c$. We consider $(a',b,c)=\mu_1(a,b,c)$. Note that $a'$ is a positive integer. It is useful to consider the polynomial $f=\lambda^2-3bc\lambda+b^2+c^2\in\mathcal{F}[\lambda]$. Note that $a,a'$ are zeros of $f$.

It is easy to see that $(a,b,c)=(2,1,1)=\mu_1(1,1,1)$ is the only solution with $b=c=1$ except for the fundamental solution. So we may assume that $(b,c)\neq(1,1)$. In this case, we have $f(b)=2b^2+c^2-3b^2c=2b^2(1-c)+c(c-b^2)<0$. Especially, $b$ must lie between the zeros of the quadratic polynomial $f$. We conclude that $a>b$ and $b>a'$. Especially, the largest entry in the triple $(a,b,c)$, namely $a$, is larger than the maximum of the triple $(a',b,c)$, namely $b$. By induction hypothesis, we can find a a sequence $(k_1,\ldots,k_r)\in\{1,2,3\}^r$ such that $(a',b,c)=(\mu_{k_r}\circ\ldots\circ\mu_{k_1})(1,1,1)$. The concatenation with $\mu_1$ yields a sequence for the triple $(a,b,c)$.
\end{proof}

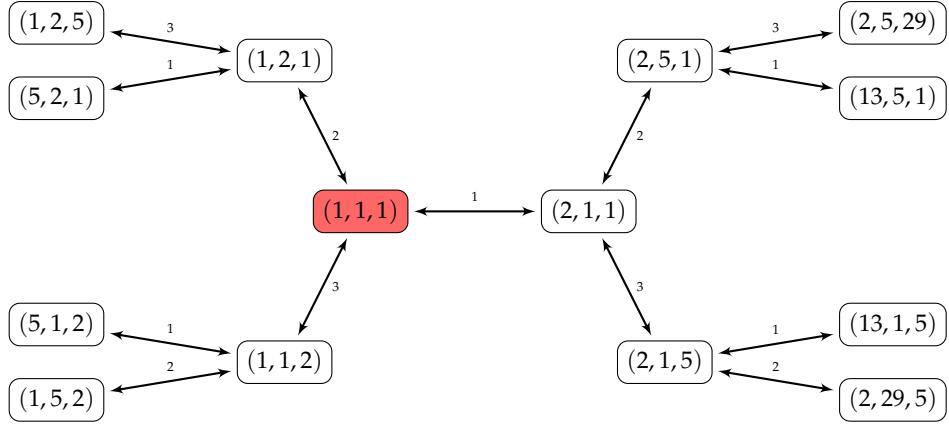
\begin{figure}[h!]
\begin{center}
\begin{tikzpicture}
\node[fill=red!60!,rectangle,rounded corners,draw] (1) at (0,0) {\small{$(1,1,1)$}};
\node[rectangle,rounded corners,draw] (2) at (3,0) {\small{$(2,1,1)$}};
\node[rectangle,rounded corners,draw] (3) at (-1,2) {\small{$(1,2,1)$}};
\node[rectangle,rounded corners,draw] (4) at (-1,-2) {\small{$(1,1,2)$}};
\node[rectangle,rounded corners,draw] (5) at (-4,1.5) {\small{$(5,2,1)$}};
\node[rectangle,rounded corners,draw] (6) at (-4,-1.5) {\small{$(5,1,2)$}};
\node[rectangle,rounded corners,draw] (7) at (-4,2.5) {\small{$(1,2,5)$}};
\node[rectangle,rounded corners,draw] (8) at (-4,-2.5) {\small{$(1,5,2)$}};
\node[rectangle,rounded corners,draw] (9) at (4,2) {\small{$(2,5,1)$}};
\node[rectangle,rounded corners,draw] (10) at (4,-2) {\small{$(2,1,5)$}};
\node[rectangle,rounded corners,draw] (11) at (7,1.5) {\small{$(13,5,1)$}};
\node[rectangle,rounded corners,draw] (12) at (7,2.5) {\small{$(2,5,29)$}};
\node[rectangle,rounded corners,draw] (13) at (7,-1.5) {\small{$(13,1,5)$}};
\node[rectangle,rounded corners,draw] (14) at (7,-2.5) {\small{$(2,29,5)$}};

\draw[<->, >=latex', shorten >=2pt, shorten <=2pt,thick,above] (1) to node {\tiny{$1$}} (2);
\draw[<->, >=latex', shorten >=2pt, shorten <=2pt,thick,right] (1) to node {\tiny{$2$}} (3);
\draw[<->, >=latex', shorten >=2pt, shorten <=2pt,thick,right] (1) to node {\tiny{$3$}} (4);
\draw[<->, >=latex', shorten >=2pt, shorten <=2pt,thick,above] (5) to node {\tiny{$1$}} (3);
\draw[<->, >=latex', shorten >=2pt, shorten <=2pt,thick,above] (6) to node {\tiny{$1$}} (4);
\draw[<->, >=latex', shorten >=2pt, shorten <=2pt,thick,above] (7) to node {\tiny{$3$}} (3);
\draw[<->, >=latex', shorten >=2pt, shorten <=2pt,thick,above] (8) to node {\tiny{$2$}} (4);
\draw[<->, >=latex', shorten >=2pt, shorten <=2pt,thick,right] (9) to node {\tiny{$2$}} (2);
\draw[<->, >=latex', shorten >=2pt, shorten <=2pt,thick,right] (10) to node {\tiny{$3$}} (2);
\draw[<->, >=latex', shorten >=2pt, shorten <=2pt,thick,above] (11) to node {\tiny{$1$}} (9);
\draw[<->, >=latex', shorten >=2pt, shorten <=2pt,thick,above] (12) to node {\tiny{$3$}} (9);
\draw[<->, >=latex', shorten >=2pt, shorten <=2pt,thick,above] (13) to node {\tiny{$1$}} (10);
\draw[<->, >=latex', shorten >=2pt, shorten <=2pt,thick,above] (14) to node {\tiny{$2$}} (10);

\end{tikzpicture}
\end{center}
\caption{Solutions of the Markov equation}
\label{SolutionsMarkov}
\end{figure}

Figure \ref{SolutionsMarkov} illustrates the mutation process. The Markov equation has a long and colorful history. In the context of cluster algebras, Beineke-Br\"ustle-Hille \cite{BBH} use the Markov equation to classify quivers with $3$ vertices of finite mutation type. The \textit{uniqueness conjecture} (c.\,f. Aigner \cite{A}) asserts that if $(a,b,c)$ and $(a,b',c')$ are two solutions of $T(a,b,c)=3$ with $a\geq b \geq c$ and $a  \geq b' \geq c'$, then $b=b'$ and $c=c'$.

\subsection{A variant}
\label{SectionVariant}

In this section we wish to study a variant of the Markov equation, which arises from the cluster algebra of a mutation-finite, non skew-symmetric matrix $B$. Felikson-Shapiro-Tumarkin \cite{FST2} have classified matrices of finite mutation type. A particular example is the $3\times 3$ matrix
\begin{align*}
B=\left(\begin{matrix}0&1&-1\\-4&0&2\\4&-2&0\end{matrix}\right).
\end{align*}
Indeed, it is easy to verify the following proposition.

\begin{prop} 
\label{MutationClassB}
For every $k\in\{1,2,3\}$ we have $\mu_1(B)=\mu_2(B)=\mu_3(B)= -B$. Especially, the mutation equivalence class of $B$ is equal to $\{B,-B\}$.
\end{prop}

Let $\mathbf{x}=(x_1,x_2,x_3)$ be the initial cluster. Assume that the seed $(R,\mathbf{y})$ is mutation equivalent to $(Q,\mathbf{x})$. Proposition \ref{MutationClassB} implies that the mutation $\mu_i$ (with $1\leq i\leq 3$) exchanges the cluster variable $y_i$ with the cluster variable $y_i'$ given by equations
\begin{align*}
&y_1y_1'=y_2^4+y_3^4,\\
&y_2y_2'=y_1+y_3^2,\\
&y_3y_3'=y_1+y_2^2.
\end{align*}
We can extend the $\mu_i$ to functions $(\mathcal{F}^{\times})^3\to(\mathcal{F}^{\times})^3$ by setting 
\begin{align*}
\mu_1(a,b,c)=(\tfrac{b^4+c^4}{a},b,c),&&\mu_2(a,b,c)=(a,\tfrac{a+c^2}{b},c),&&\mu_3(a,b,c)=(a,b,\tfrac{a+b^2}{c})
\end{align*}
for all $(a,b,c)\in(\mathcal{F}^{\times})^3$. Note that the functions $\mu_i$ are involutions.

It is easy to see that $v^T=(2,1,1)$ is the unique (up to a scalar multiple) solution to Grabowski's grading equation $v^TB=0$. By Proposition \ref{MutationClassB}, it is a unique solution to the grading equation for every signed adjacency matrix in a seed of $\mathcal{A}(B,\mathbf{x})$. That means, $\mathcal{A}(B,\mathbf{x})$ is a graded algebra where the degree of every cluster variable is equal to either $1$ or $2$. For example, the above exchange relations are homogeneous of degrees $4$, $2$ and $2$. Especially, $\mathcal{A}(Q,\mathbf{x})$ is a \textit{positively} graded algebra.

Again, we can find an invariant of the mutation rule. Define a map $T\colon(\mathcal{F}^{\times})^3\to\mathcal{F}$ by 
\begin{align*}
T(a,b,c)=\frac{a^2+b^4+c^4+2ab^2+2ac^2}{ab^2c^2}
\end{align*}for all $(a,b,c)\in(\mathcal{F}^{\times})^3$.  

\begin{prop}
\label{InvariantB}
We have $(T\circ\mu_i)(a,b,c)=T(a,b,c)$ for all $i\in\{1,2,3\}$ and $a,b,c\in\mathcal{F}^{\times}$. 
\end{prop}

\begin{proof}
Let $a,b,c\in\mathcal{F}$ be non-zero elements. Put $T=T(a,b,c)$. First assume that $i=1$. By assumption, $a$ is a zero of the quadratic polynomial $\lambda^2+(2b^2+2c^2-b^2c^2T)\lambda+b^4+c^4\in\mathcal{F}[\lambda]$. By Vieta's formula, $a'=\tfrac{b^4+c^4}{a}$ is another zero of the polynomial, so that $T(a',b,c)=T$.

Now assume that $i=2$. By assumption, $b$ is a zero of the biquadratic polynomial $\lambda^4+(2-c^2T)a\lambda^2+(a+c^2)^2\in\mathcal{F}[\lambda]$. By Vieta's formula, $b'=\tfrac{a+c^2}{b}$ is another zero of the polynomial, so that $T(a',b,c)=T$. The case $i=3$ is proved similarly.
\end{proof}

The proof shows that the mutation $aa'=b^4+c^4$ may be seen as Vieta's formula. It may be replaced by Vieta's formula $a+a'=b^2c^2T-2b^2-2c^2$. Especially, if $a,b,c,T\in\mathbb{N}\subseteq\mathcal{F}$ are positive integers, then so is $a'$. The same statement is true for $b'$ and $c'$ as well: $T\in\mathbb{Z}$ implies $b^2\mid (a+c^2)^2$ so that $b\mid a+c^2$ in this case. 

Propositions \ref{MutationClassB} and \ref{InvariantB} imply that $T$ is invariant under mutation. That is, if $(R,\mathbf{y})$ is another seed of $\mathcal{A}(Q,\mathbf{x})$, then $T(\mathbf{x})=T(\mathbf{y})$. Especially, we have $T_0=T(\mathbf{x})\in\overline{\mathcal{A}}(B,\mathbf{x})$. As before, the positive grading of $\mathcal{A}(B,\mathbf{x})$ shows that $T_0\notin \mathcal{A}(B,\mathbf{x})$, because $T_0ab^2c^2$ can not be homogeneous of degree $4$. Especially, $\mathcal{A}(B,\mathbf{x})\neq\overline{\mathcal{A}}(B,\mathbf{x})$ is not equal to its upper cluster algebra. Moreover, using Gross-Hacking-Keel-Kontsevich's linear independence of cluster monomials \cite{CKLP} and the same arguments as in Keller's proof for the Markov cluster algebra \cite{Ke}, one can show that $\mathcal{A}(B,\mathbf{x})$ is not noetherian.

Note that $T(1,1,1)=7$. The next theorem describes all triples $(a,b,c)\in\mathbf{N}^3$ of natural numbers with $T(a,b,c)=7$. Equivalently, we solve the Diophantine equation $a^2+b^4+c^4+2ab^2+2ac^2=7ab^2c^2$.

\begin{theorem}
Assume that $a,b,c$ are positive integers with $T(a,b,c)=7$. Then there exists a sequence $(k_1,\ldots,k_r)\in\{1,2,3\}^r$ of length $r\geq 0$ such that $(a,b,c)=(\mu_{k_r}\circ\ldots\circ\mu_{k_1})(1,1,1)$. 
\end{theorem}

To prove the theorem, we introduce another Diophantine equation. We can solve the equation by the same methods that we used to solve the Markov equation. Here, we denote by $\square=\{1,4,9,\ldots\}$ the set of positive perfect squares.

\begin{lemma} For $1\leq i\leq 3$ define maps $\tau_i\colon(\mathcal{F}^{\times})^3\to(\mathcal{F}^{\times})^3$ by 
\begin{align*}
\tau_1(A,B,C)=(\tfrac{B^2+C^2}{A},B,C),&&\tau_2(A,B,C)=(A,\tfrac{(A+C)^2}{B},C),&&\tau_3(A,B,C)=(A,B,\tfrac{(A+B)^2}{C})
\end{align*}
for all $A,B,C\in\mathcal{F}^{\times}$. Assume that $(A,B,C)$ are positive integers such that 
\begin{align}
\label{LemmaEqn}
A^2+B^2+C^2+2AB+2AC=7ABC. 
\end{align}
Then there exists a sequence $(k_1,\ldots,k_r)\in\{1,2,3\}^r$ of length $r\geq 0$ such that $(A,B,C)=(\tau_{k_r}\circ\ldots\circ\tau_{k_1})(1,1,1)$ and $(\tau_{k_j}\circ\ldots\circ\tau_{k_1})(1,1,1)\in\mathbb{N}\times \square^{2}$ for every $1\leq j\leq r$.
\end{lemma}

\begin{proof}[Proof of the lemma]
We prove the statement by induction on the maximum $m$ of $A,B,C$. The claim is true for $m=1$. Assume that $m>2$. Without loss of generality we may assume $B\geq C$.

First assume that $A\geq B$. It is easy to see that $(A,B,C)=(2,1,1)=\tau_1(1,1,1)$ is the only solution of the equation with $(B,C)=(1,1)$ which is different from the solution $(1,1,1)$. So we may assume that $(B,C)\neq(1,1)$. It is useful to consider the polynomial $f=\lambda^2+(2B+2C-7BC)\lambda+B^2+C^2\in\mathcal{F}[\lambda]$. Note that $A$ is a zero of $f$. Let us put $(A',B,C)=\tau_1(A,B,C)$. By construction, $A'$ is positive, and by Vieta's formula, $A'$ is another zero of $f$. Moreover, Vieta's formula $A+A'=7BC-2B-2C$ implies that $A'$ is integer. The assumptions $(B,C)\neq(1,1)$ and $B\geq C$ imply
\begin{align*}
f(B)=4B^2+C^2+2BC-7B^2C=4B^2(1-C)+C(C-B^2)+2BC(1-B)<0,
\end{align*}
so that $B$ must lie between the zeros $A,A'$ of $f$. It follows that $A>B>A'$. Especially, the largest entry in the triple $(A',B,C)$ is strictly smaller than the largest entry in $(A,B,C)$. By induction hypothesis, we can find a a sequence $(k_1,\ldots,k_r)\in\{1,2,3\}^r$ such that $(A',B,C)=(\tau_{k_r}\circ\ldots\circ\tau_{k_1})(1,1,1)$. The concatenation with $\tau_1$ yields a sequence for the triple $(A,B,C)$.

Now assume that $A<B$. It is easy to see that $(A,B,C)=(1,4,1)=\tau_2(1,1,1)$ is the only solution of the equation with $(A,C)=(1,1)$ which is different from the solution $(1,1,1)$. So we may assume $(A,C)\neq(1,1)$. It is useful to consider the polynomial $f=\lambda^2+(2A-7AC)\lambda+(A+C)^2\in\mathcal{F}[\lambda]$. Note that $B$ is a zero of $f$. Let us put $(A,B',C)=\tau_2(A,B,C)$. By construction, $B'$ is positive, and by Vieta's formula, $B'$ is another zero of $f$. Moreover, Vieta's formula $B+B'=7AC-2A$ implies that $B'$ is integer. The assumption $(A,C)\neq(1,1)$ implies
\begin{align*}
&f(A)=4A^2(1-C)+C(C-A^2)+2AC(1-A)<0,&&\textrm{if }C\leq A,\\
&f(C)=2C^2(1-A)+A(A-C^2)+4AC(1-C)<0,&&\textrm{if }A\leq C.
\end{align*}
In the case $C\leq A$, the natural number $A$ must lie between the two zeros of $f$, so that $B'<A<B$. Moreover, $C\leq A$ implies $C<B$, so that in this case $\operatorname{max}(A,B,C)=B>A,B',C$. In the case $A\leq C$, the natural number $C$ must be different from $B$ and must lie between the zeros of $f$, so that $B'<C<B$. Therefore, in this case we also have $\operatorname{max}(A,B,C)=B>A,B',C$. As above, the induction hypothesis implies that we can find a a sequence $(k_1,\ldots,k_r)\in\{1,2,3\}^r$ such that $(A,B',C)=(\mu_{k_r}\circ\ldots\circ\mu_{k_1})(1,1,1)$. Moreover, $B'$ is a perfect square. It follows that $B=(A+C)^2/B'$ is a perfect square and can find a sequence $(k_1,\ldots,k_{r+1})\in\{1,2,3\}^{r+1}$ such that $(A,B,C)=(\tau_{k_{r+1}}\circ\ldots\circ\tau_{k_1})(1,1,1)$ by setting $k_{r+1}=2$.  
\end{proof}

We are ready to prove the theorem.

\begin{proof}[Proof of the theorem]
Suppose that the positive integers $a,b,c$ satisfy $T(a,b,c)=7$. Then the triple $(A,B,C)=(a,b^2,c^2)$ is a solution to Equation \ref{LemmaEqn}. By the lemma, there is a sequence $(k_1,\ldots,k_r)$ of indices such that $(a,b^2,c^2)=(\tau_{k_r}\circ\ldots\circ\tau_{k_1})(1,1,1)$ and $(A_i,B_i,C_i)=(\tau_{k_i}\circ\ldots\circ\tau_{k_1})(1,1,1)\in\mathbb{N}\times\square^{2}$ for every $0\leq i\leq r$. Substitute $A_i=a_i$, $B_i=b_i^2$ and $C_i=c_i^2$ with positive integers $b_i,c_i$ for all $i$. Then $\tau_{k_{i+1}}(a_i,b_i^2,c_i^2)=(a_{i+1},b_{i+1}^2,c_{i+1}^2)$ implies $\mu_{k_{i+1}}(a_i,b_i,c_i)=(a_{i+1},b_{i+1},c_{i+1})$, so that $(a,b,c)=(\mu_{k_r}\circ\ldots\circ\mu_{k_1})(1,1,1)$.\end{proof}

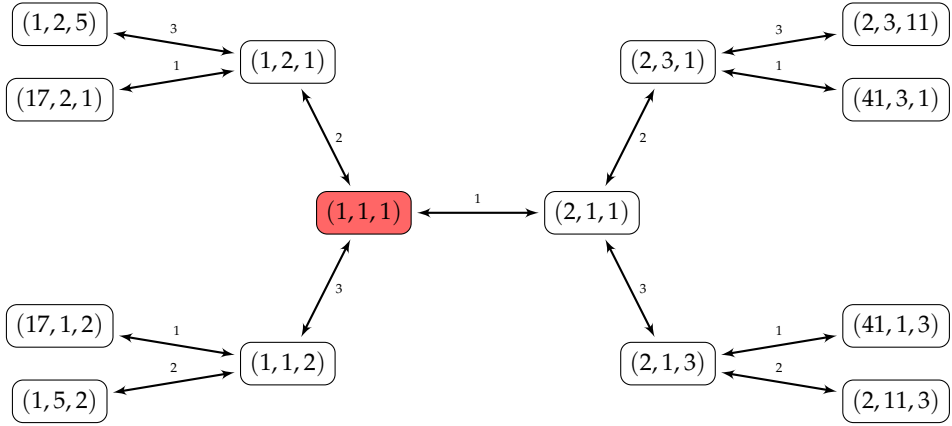
\begin{figure}[h!]
\begin{center}
\begin{tikzpicture}
\node[fill=red!60!,rectangle,rounded corners,draw] (1) at (0,0) {\small{$(1,1,1)$}};
\node[rectangle,rounded corners,draw] (2) at (3,0) {\small{$(2,1,1)$}};
\node[rectangle,rounded corners,draw] (3) at (-1,2) {\small{$(1,2,1)$}};
\node[rectangle,rounded corners,draw] (4) at (-1,-2) {\small{$(1,1,2)$}};
\node[rectangle,rounded corners,draw] (5) at (-4,1.5) {\small{$(17,2,1)$}};
\node[rectangle,rounded corners,draw] (6) at (-4,-1.5) {\small{$(17,1,2)$}};
\node[rectangle,rounded corners,draw] (7) at (-4,2.5) {\small{$(1,2,5)$}};
\node[rectangle,rounded corners,draw] (8) at (-4,-2.5) {\small{$(1,5,2)$}};
\node[rectangle,rounded corners,draw] (9) at (4,2) {\small{$(2,3,1)$}};
\node[rectangle,rounded corners,draw] (10) at (4,-2) {\small{$(2,1,3)$}};
\node[rectangle,rounded corners,draw] (11) at (7,1.5) {\small{$(41,3,1)$}};
\node[rectangle,rounded corners,draw] (12) at (7,2.5) {\small{$(2,3,11)$}};
\node[rectangle,rounded corners,draw] (13) at (7,-1.5) {\small{$(41,1,3)$}};
\node[rectangle,rounded corners,draw] (14) at (7,-2.5) {\small{$(2,11,3)$}};

\draw[<->, >=latex', shorten >=2pt, shorten <=2pt,thick,above] (1) to node {\tiny{$1$}} (2);
\draw[<->, >=latex', shorten >=2pt, shorten <=2pt,thick,right] (1) to node {\tiny{$2$}} (3);
\draw[<->, >=latex', shorten >=2pt, shorten <=2pt,thick,right] (1) to node {\tiny{$3$}} (4);
\draw[<->, >=latex', shorten >=2pt, shorten <=2pt,thick,above] (5) to node {\tiny{$1$}} (3);
\draw[<->, >=latex', shorten >=2pt, shorten <=2pt,thick,above] (6) to node {\tiny{$1$}} (4);
\draw[<->, >=latex', shorten >=2pt, shorten <=2pt,thick,above] (7) to node {\tiny{$3$}} (3);
\draw[<->, >=latex', shorten >=2pt, shorten <=2pt,thick,above] (8) to node {\tiny{$2$}} (4);
\draw[<->, >=latex', shorten >=2pt, shorten <=2pt,thick,right] (9) to node {\tiny{$2$}} (2);
\draw[<->, >=latex', shorten >=2pt, shorten <=2pt,thick,right] (10) to node {\tiny{$3$}} (2);
\draw[<->, >=latex', shorten >=2pt, shorten <=2pt,thick,above] (11) to node {\tiny{$1$}} (9);
\draw[<->, >=latex', shorten >=2pt, shorten <=2pt,thick,above] (12) to node {\tiny{$3$}} (9);
\draw[<->, >=latex', shorten >=2pt, shorten <=2pt,thick,above] (13) to node {\tiny{$1$}} (10);
\draw[<->, >=latex', shorten >=2pt, shorten <=2pt,thick,above] (14) to node {\tiny{$2$}} (10);

\end{tikzpicture}
\end{center}
\caption{Solutions of the Diophantine equation}
\label{SolutionsSkew}
\end{figure}

Figure \ref{SolutionsSkew} illustrates the mutation process. Note that the uniqueness conjecture fails for this Diophantine equation. For example, $(a,b,c)=(41,3,1)$ and $(a',b',c')=(41,14,1)$ are two solutions of $T(a,b,c)=7$ with $a\geq b \geq c$ and $a' \geq b' \geq c'$, but $b\neq b'$.

\subsection{A Diophantine equation from the torus minus a disk}

\label{SectionTorusMinus}
\begin{figure}
\begin{center}
\begin{tikzpicture}

\newcommand\radius{0.3}
\newcommand\grid{1.2}
\newcommand\uber{0.3}

\draw[blue,fill=black!30!] (\grid,\grid-\radius) circle (\radius);
\draw[blue,fill=black!30!] (3*\grid,\grid-\radius) circle (\radius);
\draw[blue,fill=black!30!] (5*\grid,\grid-\radius) circle (\radius);
\draw[blue,fill=black!30!] (\grid,3*\grid-\radius) circle (\radius);
\draw[blue,fill=black!30!] (3*\grid,3*\grid-\radius) circle (\radius);
\draw[blue,fill=black!30!] (5*\grid,3*\grid-\radius) circle (\radius);
\draw[blue,fill=black!30!] (\grid,5*\grid-\radius) circle (\radius);
\draw[blue,fill=black!30!] (3*\grid,5*\grid-\radius) circle (\radius);
\draw[blue,fill=black!30!] (5*\grid,5*\grid-\radius) circle (\radius);

\draw[color=black!30!] (0,0-\uber) to (0,6*\grid+\uber);
\draw[color=black!30!] (2*\grid,0-\uber) to (2*\grid,6*\grid+\uber);
\draw[color=black!30!] (4*\grid,0-\uber) to (4*\grid,6*\grid+\uber);
\draw[color=black!30!] (6*\grid,0-\uber) to (6*\grid,6*\grid+\uber);
\draw[color=black!30!] (0-\uber,0) to (6*\grid+\uber,0);
\draw[color=black!30!] (0-\uber,2*\grid) to (6*\grid+\uber,2*\grid);
\draw[color=black!30!] (0-\uber,4*\grid) to (6*\grid+\uber,4*\grid);
\draw[color=black!30!] (0-\uber,6*\grid) to (6*\grid+\uber,6*\grid);

\coordinate(AA) at (\grid,\grid) {};
\coordinate(AB) at (3*\grid,\grid) {};
\coordinate(AC) at (5*\grid,\grid) {};
\coordinate(BA) at (\grid,3*\grid) {};
\coordinate(BB) at (3*\grid,3*\grid) {};
\coordinate(BC) at (5*\grid,3*\grid) {};
\coordinate(CA) at (\grid,5*\grid) {};
\coordinate(CB) at (3*\grid,5*\grid) {};
\coordinate(CC) at (5*\grid,5*\grid) {};

\draw[bend left=10,color=red!60!] (AA) to (AB);
\draw[bend left=10,color=red!60!] (AB) to (AC);
\draw[bend left=10,color=red!60!] (BA) to (BB);
\draw[bend left=10,color=red!60!] (BB) to (BC);
\draw[bend left=10,color=red!60!] (CA) to (CB);
\draw[bend left=10,color=red!60!] (CB) to (CC);

\draw[bend right=40,color=red!60!,densely dotted,thick] (AA) to (2*\grid,2*\grid);
\draw[bend left=40,color=red!60!,densely dotted,thick] (2*\grid,2*\grid) to (BB);
\draw[bend right=40,color=red!60!,densely dotted,thick] (BB) to (4*\grid,4*\grid);
\draw[bend left=40,color=red!60!,densely dotted,thick] (4*\grid,4*\grid) to (CC);
\draw[bend right=40,color=red!60!,densely dotted,thick] (BA) to (2*\grid,4*\grid);
\draw[bend left=40,color=red!60!,densely dotted,thick] (2*\grid,4*\grid) to (CB);
\draw[bend right=40,color=red!60!,densely dotted,thick] (AB) to (4*\grid,2*\grid);
\draw[bend left=40,color=red!60!,densely dotted,thick] (4*\grid,2*\grid) to (BC);

\draw[bend left=90,color=red!60!,dashed] (AA) to (BA);
\draw[bend left=90,color=red!60!,dashed] (BA) to (CA);
\draw[bend left=90,color=red!60!,dashed] (AB) to (BB);
\draw[bend left=90,color=red!60!,dashed] (BB) to (CB);
\draw[bend left=90,color=red!60!,dashed] (AC) to (BC);
\draw[bend left=90,color=red!60!,dashed] (BC) to (CC);

\draw[bend right=90,color=red!60!,dotted,thick] (AA) to (BA);
\draw[bend right=90,color=red!60!,dotted,thick] (BA) to (CA);
\draw[bend right=90,color=red!60!,dotted,thick] (AB) to (BB);
\draw[bend right=90,color=red!60!,dotted,thick] (BB) to (CB);
\draw[bend right=90,color=red!60!,dotted,thick] (AC) to (BC);
\draw[bend right=90,color=red!60!,dotted,thick] (BC) to (CC);

\node at (AA) {\Large{$\cdot$}};
\node at (AB) {\Large{$\cdot$}};
\node at (AC) {\Large{$\cdot$}};
\node at (BA) {\Large{$\cdot$}};
\node at (BB) {\Large{$\cdot$}};
\node at (BC) {\Large{$\cdot$}};
\node at (CA) {\Large{$\cdot$}};
\node at (CB) {\Large{$\cdot$}};
\node at (CC) {\Large{$\cdot$}};

\end{tikzpicture}
\end{center}
\caption{A triangulations of the torus minus a disk with one marked point on the boundary}
\label{Figure:TriangulationWithBdy}
\end{figure}
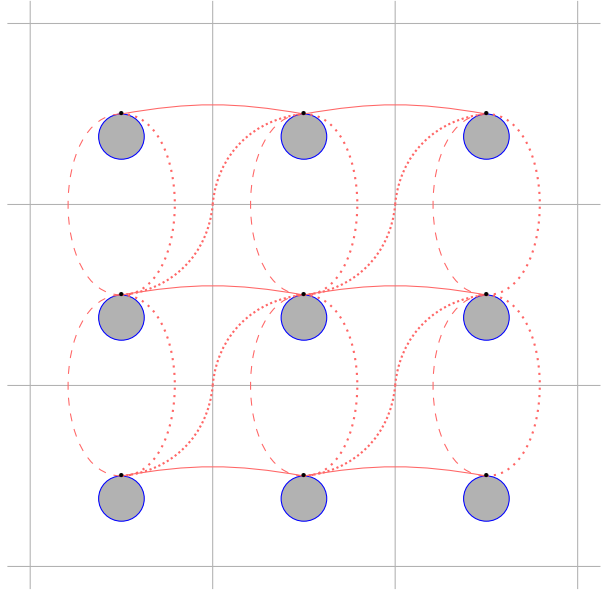

Let $T=S^1\times S^1$ be the torus and $D^2\subset T$ a disk. We consider the bordered surface $\Sigma=T\backslash D^2$. Let $P\in \partial S=\partial D^2$ be a point on the boundary, and put $M=\{P\}$. Then $\mathcal{A}(\Sigma,M)$ is a cluster algebra of rank $4$ with four mutable initial cluster variables $x_1,x_2,x_3,x_4$ and one frozen variable $x_5$. 

We also consider the cluster algebra $\mathcal{A}'(\Sigma,M)$ which is obtained from $\mathcal{A}(\Sigma,M)$ by ignoring the frozen variable. Figure \ref{Figure:TriangulationWithBdy} visualizes a triangulation $\mathcal{T}$ of the bordered surface. Figure \ref{InitialQuiver} displays quivers $Q=Q(\mathcal{T})$ and $Q'$ of initial seeds for $\mathcal{A}(\Sigma,M)$ and $\mathcal{A}'(\Sigma,M)$, respectively. Mutable vertices are colored red and the frozen vertex is colored blue.  

\begin{prop} 
\label{MutationClassTorus}
We have $Q\simeq \mu_1(Q)\simeq \mu_2(Q)\simeq \mu_3(Q)\simeq \mu_4(Q)$. Especially, the mutation equivalence class of $Q$ is a singleton. The same is true for the quiver $Q'$.
\end{prop}

\begin{proof}
It is easy to check that the following bijections on vertex sets induce isomorphisms of quivers. 
\begin{center}
\begin{tabular}{|c|c|c|c|c|}\hline
$Q_0$&$\mu_1(Q)_0$&$\mu_2(Q)_0$&$\mu_3(Q)_0$&$\mu_4(Q)_0$\\\hline
1&2&2&4&2\\
2&1&1&1&3\\
3&3&3&2&4\\
4&4&4&3&1\\
5&5&5&5&5\\\hline
\end{tabular}
\end{center}
The corresponding table for $Q'$ is the same table with the last row removed.\end{proof}

Let $\mathbf{x}=(x_1,x_2,x_3,x_4,x_5)$ be the initial cluster. Assume that the seed $(R,\mathbf{y})$ is mutation equivalent to $(Q,\mathbf{x})$. Without loss of generality, let us assume that the vertices of $R$ are indexed such that the isomorphism $R\simeq Q$ is the identity on $Q_0=R_0=\{1,2,3,4,5\}$. Proposition \ref{MutationClassTorus} implies that the mutation $\mu_i$ (with $1\leq i\leq 4$) exchanges the cluster variable $y_i$ with the cluster variable $y_i'$ given by equations
\begin{align*}
&y_1y_1'=y_2^2+y_3y_4,&&y_2y_2'=y_1^2+y_3y_4,\\
&y_3y_3'=y_1y_5+y_2y_4,&&y_4y_4'=y_1y_3+y_2y_5.
\end{align*}
We can extend the mutation and reordering process to functions $(\mathcal{F}^{\times})^5\to(\mathcal{F}^{\times})^5$ by setting
\begin{align*}
&\mu_1(a,b,c,d,e)=\left(b,\tfrac{b^2+cd}{a},c,d,e\right),&&\mu_2(a,b,c,d,e)=\left(\tfrac{a^2+cd}{b},a,c,d,e\right),\\
&\mu_3(a,b,c,d,e)=\left(d,a,b,\tfrac{ae+bd}{c},e\right),&&\mu_4(a,b,c,d,e)=\left(b,c,\tfrac{ac+be}{d},a,e\right).
\end{align*}
for all $(a,b,c,d,e)\in(\mathcal{F}^{\times})^5$. Note that we have $\mu_1=\mu_2^{-1}$ and $\mu_3=\mu_4^{-1}$.

Every mutable vertex of $Q$ is incident to exactly two incoming and two outgoing arrows. (Especially, the quiver is contained in Ladkani's list \cite{L} of quivers for which the number of arrows is invariant under mutation.) Thus $v^T=(1,1,1,1,1)$ is a solution to Grabowski's grading equation $v^TB=0$. By Proposition \ref{MutationClassMarkov}, the same is true for the signed adjacency matrix $B(Q')$ for every quiver $Q'$ which is mutation equivalent to $Q$. Hence, the cluster algebra is graded such that the degree of every cluster variable is equal to $1$. An exchange relation in $\mathcal{A}(\Sigma,M)$ has the form $x_{\alpha}x_{\gamma}+x_{\beta}x_{\delta}=x_{\epsilon}x_{\digamma}$ for some (not necessarily distinct) cluster or frozen variables $x_{\alpha},x_{\beta},x_{\gamma},x_{\delta},x_{\epsilon},x_{\digamma}$ and is hence homogeneous of degree $2$.

\begin{figure}[h!]
\begin{center}
\begin{tikzpicture}[scale=0.8]
\node[fill=red!60!,rectangle,rounded corners,draw] (1) at (-1,0) {$1$};
\node[fill=red!60!,rectangle,rounded corners,draw] (2) at (2,0) {$2$};
\node[fill=red!60!,rectangle,rounded corners,draw] (3) at (4,2) {$3$};
\node[fill=red!60!,rectangle,rounded corners,draw] (4) at (4,-2){$4$};
\node[fill=blue!30!,rectangle,rounded corners,draw] (5) at (6,0) {$5$};

\draw[->, >=latex', shorten >=2pt, shorten <=2pt, bend right=10,thick] (2) to (1);
\draw[->, >=latex', shorten >=2pt, shorten <=2pt, bend left=10,thick] (2) to (1);
\draw[->, >=latex', shorten >=2pt, shorten <=2pt,thick] (3) to (2);
\draw[->, >=latex', shorten >=2pt, shorten <=2pt,thick] (4) to (2);
\draw[->, >=latex', shorten >=2pt, shorten <=2pt,thick] (3) to (4);
\draw[->, >=latex', shorten >=2pt, shorten <=2pt, bend left=10,thick] (1) to (3);
\draw[->, >=latex', shorten >=2pt, shorten <=2pt, bend right=10,thick] (1) to (4);
\draw[->, >=latex', shorten >=2pt, shorten <=2pt,thick] (4) to (5);
\draw[->, >=latex', shorten >=2pt, shorten <=2pt,thick] (5) to (3);

\node[fill=red!60!,rectangle,rounded corners,draw] (1b) at (8,0) {$1$};
\node[fill=red!60!,rectangle,rounded corners,draw] (2b) at (11,0) {$2$};
\node[fill=red!60!,rectangle,rounded corners,draw] (3b) at (13,2) {$3$};
\node[fill=red!60!,rectangle,rounded corners,draw] (4b) at (13,-2){$4$};

\draw[->, >=latex', shorten >=2pt, shorten <=2pt, bend right=10,thick] (2b) to (1b);
\draw[->, >=latex', shorten >=2pt, shorten <=2pt, bend left=10,thick] (2b) to (1b);
\draw[->, >=latex', shorten >=2pt, shorten <=2pt,thick] (3b) to (2b);
\draw[->, >=latex', shorten >=2pt, shorten <=2pt,thick] (4b) to (2b);
\draw[->, >=latex', shorten >=2pt, shorten <=2pt,thick] (3b) to (4b);
\draw[->, >=latex', shorten >=2pt, shorten <=2pt, bend left=10,thick] (1b) to (3b);
\draw[->, >=latex', shorten >=2pt, shorten <=2pt, bend right=10,thick] (1b) to (4b);
\end{tikzpicture}
\end{center}
\caption{The quivers $Q$ and $Q'$ of initial seeds of $\mathcal{A}$ and $\mathcal{A}'$}
\label{InitialQuiver}
\end{figure}
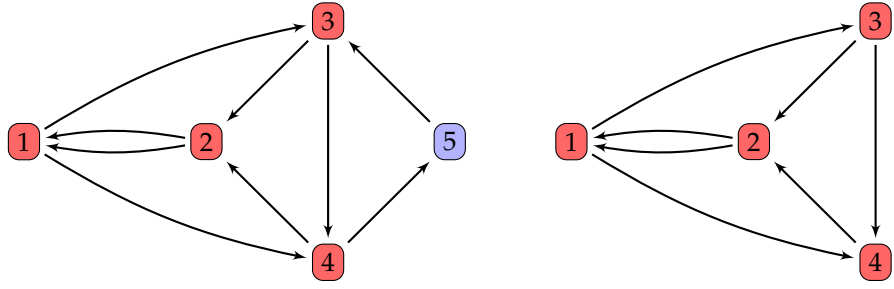

Again, we can find an invariant of the mutation rule. Define a map $T\colon(\mathcal{F}^{\times})^5\to \mathcal{F}$ by sending $(u_1,u_2,u_3,u_4,u_5)$ to the rational expression 
\begin{align*}
\frac{u_1u_2u_5^2+u_1^2u_4u_5+u_3u_4^2u_5+u_2^2u_3u_5+u_1u_2u_4^2+u_2^2u_4u_5+u_1u_2u_3^2+u_1^2u_3u_5+u_3^2u_4u_5}{u_1u_2u_3u_4}\in\mathcal{F}.
\end{align*}
Viewed as an element in the skein algebra $\operatorname{SK}_0(S,P)\supseteq\mathcal{A}(\Sigma,M)$, see Muller \cite[Section 1.3]{M2}, the element $T(x_1,x_2,x_3,x_4,x_5)$ is related to the curve homotopic to the boundary of $S$, i.e. the loop around the disk $D^2$.

One can use Maple to check the following proposition.
\begin{prop}
We have $(T\circ \mu_i)(a,b,c,d,e)=T(a,b,c,d,e)$ for all $i\in\{1,2,3,4,5\}$ and all $a,b,c,d,e\in\mathcal{F}^{\times}$.
\end{prop}

We finish with some questions:
\begin{que}
\begin{itemize}
\item[(a)] Which solutions $(a,b,c,d,1)\in\mathbb{N}^5$ of the (inhomogeneous) Diophantine equation $T(a,b,c,d,1)=9$ can be obtained from the initial solution $(1,1,1,1,1)$ by a sequence of cluster mutations? See Figure \ref{Solutions} for some examples (where we ignore the entry $e=1$). Which solutions $(a,b,c,d,e)\in\mathbb{N}^5$ of the (homogeneous) Diophantine equation $T(a,b,c,d,e)=9$ can be obtained from $(e,e,e,e,e)$ by a sequence of cluster transformations?  Computer calculations suggest that we do not get all solutions by cluster transformations.
\item[(b)] The maps $\mu_1$ and $\mu_2$ are related to Vieta's formula. More precisely, let us define maps $(\mathbb{Q}^{+})^5\to(\mathbb{Q}^{+})^5$ as the twists $t_1(a,b,c,d,e)=(b,a,c,d,e)$ and $t_2(a,b,c,d,e)=(a,b,d,c,e)$ and as the Vieta functions 
\begin{align*}
&v_1(a,b,c,d,e)=\left(\tfrac{b^2+cd}{a},b,c,d,e\right),&&v_2(a,b,c,d,e)=\left(a,\tfrac{a^2+cd}{b},c,d,e\right),\\
&v_3(a,b,c,d,e)=\left(a,b,\tfrac{ab(d^2+e^2)+de(a^2+b^2)}{c(ab+de)},d,e\right),&&v_4(a,b,c,d,e)=\left(a,b,c,\tfrac{ab(c^2+e^2)+ce(a^2+b^2)}{d(ab+ce)},e\right).
\end{align*}
for all $(a,b,c,d,e)\in(\mathbb{Q}^{+})^5$. Then a simple calculation shows that $(T\circ \varphi)(a,b,c,d,e)=T(a,b,c,d,e)$ for all maps $\varphi\in\{v_1,v_2,v_3,v_4,t_1,t_2\}$ and all $a,b,c,d,e\in\mathcal{F}^{\times}$. Moreover, the maps satisfy the relations
\begin{align*}
&v_1t_1=t_1v_2=\mu_2, &&v_3t_1=t_1v_3,\\
&v_1t_2=t_2v_1,&&v_3t_2=t_2v_4\\
&v_2t_1=t_1v_1=\mu_1,&&v_4t_1=t_1v_4\\
&v_2t_2=t_2v_2,&&v_4t_2=t_2v_3.\\
&t_2t_1\mu_3=\mu_4t_1t_2,
\end{align*}
Thus $v_2=t_1v_1t_1$, $v_4=t_2v_3t_2$, $\mu_1=t_1v_1$, $\mu_2=t_1v_2$, $\mu_4=t_1t_2\mu_3t_1t_2$ so that the group generated by the functions is generated as $G=\langle v_1,v_3,\mu_3,t_1,t_2\rangle$. Which solutions $(a,b,c,d,e)\in\mathbb{N}^5$ to the (homogeneous) Diophantine equation $T(a,b,c,d,e)=9$ be obtained from $(1,1,1,1,1)$ by applying $G$ and rescaling possible rational entries? Computer calculations suggest that we do not get all solutions by such transformations.
\end{itemize}
\end{que}

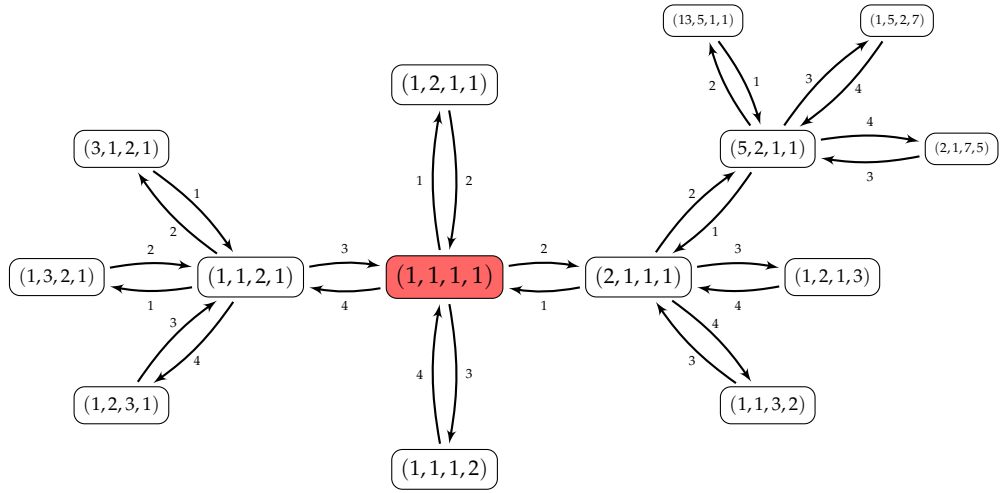
\begin{figure}[h!]
\begin{center}
\begin{tikzpicture}[scale=0.85]
\node[fill=red!60!,rectangle,rounded corners,draw] (1) at (0,0) {\small{$(1,1,1,1)$}};
\node[rectangle,rounded corners,draw] (2) at (3,0) {\footnotesize{$(2,1,1,1)$}};
\node[rectangle,rounded corners,draw] (3) at (0,3) {\footnotesize{$(1,2,1,1)$}};
\node[rectangle,rounded corners,draw] (4) at (-3,0) {\footnotesize{$(1,1,2,1)$}};
\node[rectangle,rounded corners,draw] (5) at (0,-3) {\footnotesize{$(1,1,1,2)$}};
\node[rectangle,rounded corners,draw] (6) at (6,0) {\scriptsize{$(1,2,1,3)$}};
\node[rectangle,rounded corners,draw] (7) at (5,-2) {\scriptsize{$(1,1,3,2)$}};
\node[rectangle,rounded corners,draw] (8) at (5,2) {\scriptsize{$(5,2,1,1)$}};
\node[rectangle,rounded corners,draw] (9) at (-6,0) {\scriptsize{$(1,3,2,1)$}};
\node[rectangle,rounded corners,draw] (10) at (-5,2) {\scriptsize{$(3,1,2,1)$}};
\node[rectangle,rounded corners,draw] (11) at (-5,-2) {\scriptsize{$(1,2,3,1)$}};
\node[rectangle,rounded corners,draw] (12) at (4,4) {\tiny{$(13,5,1,1)$}};
\node[rectangle,rounded corners,draw] (13) at (7,4) {\tiny{$(1,5,2,7)$}};
\node[rectangle,rounded corners,draw] (14) at (8,2) {\tiny{$(2,1,7,5)$}};

\draw[->, >=latex', shorten >=2pt, shorten <=2pt, bend left=10,thick,above] (1) to node {\tiny{$2$}} (2);
\draw[->, >=latex', shorten >=2pt, shorten <=2pt, bend left=10,thick,below] (2) to node {\tiny{$1$}} (1);
\draw[->, >=latex', shorten >=2pt, shorten <=2pt, bend left=10,thick,left] (1) to node {\tiny{$1$}} (3);
\draw[->, >=latex', shorten >=2pt, shorten <=2pt, bend left=10,thick,right] (3) to node {\tiny{$2$}} (1);
\draw[->, >=latex', shorten >=2pt, shorten <=2pt, bend left=10,thick,below] (1) to node {\tiny{$4$}} (4);
\draw[->, >=latex', shorten >=2pt, shorten <=2pt, bend left=10,thick,above] (4) to node {\tiny{$3$}} (1);
\draw[->, >=latex', shorten >=2pt, shorten <=2pt, bend left=10,thick,right] (1) to node {\tiny{$3$}} (5);
\draw[->, >=latex', shorten >=2pt, shorten <=2pt, bend left=10,thick,left] (5) to node {\tiny{$4$}} (1);
\draw[->, >=latex', shorten >=2pt, shorten <=2pt, bend left=10,thick,above] (2) to node {\tiny{$3$}} (6);
\draw[->, >=latex', shorten >=2pt, shorten <=2pt, bend left=10,thick,below] (6) to node {\tiny{$4$}} (2);
\draw[->, >=latex', shorten >=2pt, shorten <=2pt, bend left=10,thick,above] (2) to node {\tiny{$4$}} (7);
\draw[->, >=latex', shorten >=2pt, shorten <=2pt, bend left=10,thick,below] (7) to node {\tiny{$3$}} (2);
\draw[->, >=latex', shorten >=2pt, shorten <=2pt, bend left=10,thick,above] (2) to node {\tiny{$2$}} (8);
\draw[->, >=latex', shorten >=2pt, shorten <=2pt, bend left=10,thick,below] (8) to node {\tiny{$1$}} (2);
\draw[->, >=latex', shorten >=2pt, shorten <=2pt, bend left=10,thick,above] (9) to node {\tiny{$2$}} (4);
\draw[->, >=latex', shorten >=2pt, shorten <=2pt, bend left=10,thick,below] (4) to node {\tiny{$1$}} (9);
\draw[->, >=latex', shorten >=2pt, shorten <=2pt, bend left=10,thick,above] (10) to node {\tiny{$1$}} (4);
\draw[->, >=latex', shorten >=2pt, shorten <=2pt, bend left=10,thick,below] (4) to node {\tiny{$2$}} (10);
\draw[->, >=latex', shorten >=2pt, shorten <=2pt, bend left=10,thick,above] (11) to node {\tiny{$3$}} (4);
\draw[->, >=latex', shorten >=2pt, shorten <=2pt, bend left=10,thick,below] (4) to node {\tiny{$4$}} (11);
\draw[->, >=latex', shorten >=2pt, shorten <=2pt, bend left=10,thick,left] (8) to node {\tiny{$2$}} (12);
\draw[->, >=latex', shorten >=2pt, shorten <=2pt, bend left=10,thick,right] (12) to node {\tiny{$1$}} (8);
\draw[->, >=latex', shorten >=2pt, shorten <=2pt, bend left=10,thick,left] (8) to node {\tiny{$3$}} (13);
\draw[->, >=latex', shorten >=2pt, shorten <=2pt, bend left=10,thick,right] (13) to node {\tiny{$4$}} (8);
\draw[->, >=latex', shorten >=2pt, shorten <=2pt, bend left=10,thick,above] (8) to node {\tiny{$4$}} (14);
\draw[->, >=latex', shorten >=2pt, shorten <=2pt, bend left=10,thick,below] (14) to node {\tiny{$3$}} (8);

\end{tikzpicture}
\end{center}
\caption{Solutions to the Diophantine equation $T(a,b,c,d,1)=9$}
\label{Solutions}
\end{figure}

\end{document}